\documentclass[a4paper,
               twoside,
               ]{amsart}

\usepackage{amsmath} 

\usepackage{amssymb} 
\usepackage[all]{xy} 
\usepackage{color} 
\usepackage[pdfstartview=FitH]{hyperref} 

\newtheorem{theorem}{Theorem}[section]

\newtheorem{lemma}[theorem]{Lemma}

\newtheorem{proposition}[theorem]{Proposition}
\theoremstyle{definition}
\newtheorem{definition}[theorem]{Definition}
\theoremstyle{remark}
\newtheorem{remark}[theorem]{Remark}

\numberwithin{equation}{section}

\newcommand*{\ab}{\mathrm{ab}} 
\newcommand*{\et}{\mathrm{\acute{e}t}} 
\newcommand*{\cyc}{\mathrm{cyc}} 

\newcommand*{\FF}{\mathbb{F}} 
\newcommand*{\Z}{\mathbb{Z}} 
\newcommand*{\Q}{\mathbb{Q}} 
\newcommand*{\mto}{\rightarrow} 
\newcommand*{\fundG}{\pi_1^{\et}} 
\newcommand*{\Sect}{\Gamma} 
\newcommand{\Sectc}{\Gamma_{\mathrm{c}}} 
\newcommand*{\tensor}{\otimes} 
\newcommand*{\ctensor}{\hat{\tensor}} 
\newcommand*{\isomorph}{\cong} 
\newcommand*{\del}{\partial} 
\newcommand*{\Order}{\mathcal{O}} 
\newcommand*{\comp}{\circ} 
\DeclareMathOperator{\Aut}{Aut} 
\DeclareMathOperator{\Gal}{Gal} 
\DeclareMathOperator{\Hom}{Hom} 
\DeclareMathOperator{\RDer}{R} 
\DeclareMathOperator{\LDer}{L} 
\DeclareMathOperator{\Spec}{Spec} 
\DeclareMathOperator{\HF}{H} 
\DeclareMathOperator{\GL}{GL} 
\DeclareMathOperator{\Jac}{Jac} 

\newcommand*{\Sg}{\mathcal{S}} 
\newcommand*{\onsg}{\lhd_{\mathrm{o}}} 
\newcommand*{\sheaf}[1]{\mathcal{#1}} 
\newcommand*{\cmplx}[1]{{#1}^\bullet} 
\newcommand*{\cat}[1]{\mathbf{#1}} 
\newcommand*{\Ar}{\mathfrak{A}} 
\newcommand*{\Cr}{\mathfrak{C}} 
\newcommand*{\Sh}{\hat{\cat{Sh}}} 
\newcommand*{\Frob}{F} 
\newcommand*{\qiso}{\simeq} 
\newcommand*{\place}[1]{\mathfrak{p}} 
\newcommand*{\X}{\mathcal{X}} 
\DeclareMathOperator{\God}{G} 
\DeclareMathOperator{\Cone}{Cone} 
\DeclareMathOperator{\id}{id} 

\begin{document}

\title{Noncommutative Main Conjectures of Geometric Iwasawa Theory}
\author{Malte Witte}
\address{Malte Witte,\newline Universit\"at Heidelberg,\newline Mathematisches Institut,\newline Im Neuenheimer Feld 288}
\email{witte@mathi.uni-heidelberg.de}
%
%

\begin{abstract}
 In this chapter we give a survey on noncommutative main conjectures of Iwasawa theory in a geometric setting, i.\,e.\ for separated schemes of finite type over a finite field, as stated and proved by Burns and the author. We will also comment briefly on versions of the main conjecture for function fields.
\end{abstract}


\maketitle

In this chapter we give a survey on noncommutative main conjectures of Iwasawa theory in a geometric setting, i.\,e.\ for a separated scheme of finite type over a finite field, as stated and proved in \cite{Witte:MCVarFF} and \cite{Burns:MCinGIwTh+RelConj}. We begin by formulating the conjecture in Section~\ref{sec:Formulation} and then give a sketch of the proof in Section~\ref{sec:Proofs}. In Section~\ref{sec:Function fields} we will comment on the special case that the scheme is smooth, geometrically connected, and of dimension $1$, which corresponds to an analogue of the main conjecture for function fields.

 \section{Formulation of the Conjecture}\label{sec:Formulation}

 As a motivation, let us begin by recalling the formulation of the main conjecture in the number field case from \cite[Thm. 5.1]{CoatesKim:Introduction}. Since the letter $p$ is conventionally reserved to denote the characteristic of the base fields appearing later in the text, we will denote by $\ell$ the prime which is conventionally denoted by $p$ in Iwasawa theory. Fix a totally real number field $F$ and an admissible $\ell$-adic Lie extension $F_{\infty}/F$ in the sense of \cite[\S 1]{CoatesKim:Introduction}. (In particular, $F_{\infty}$ is also totally real.) Let $\kappa_F$ denote the cyclotomic character.

 \begin{theorem}\label{thm:mc number field}
 Let $F_{\infty}/F$ be unramified outside the finite set of primes $\Sigma$ (with $\ell\in \Sigma$) and assume its Galois group $G$ contains no element of order $\ell$. If the generalized Iwasawa conjecture holds for $F_{\infty}/F$, then there exists a $\zeta_{F_{\infty}/F}\in K_1(\Lambda(G)_S)$ satisfying the interpolation property
 $$
 \zeta_{F_{\infty}/F}(\rho\kappa_F^n)=L_{\Sigma}(\rho,1-n)
 $$
 for all Artin representations $\rho$ of $G$ and the equation
 $$
 \del \zeta_{F_{\infty}/F}=[\X(F_{\infty})]-[\Z_{\ell}]
 $$
 in  $K_1(\Lambda(G),\Lambda(G)_S)$.
 \end{theorem}

 Our aim is to formulate an analogue of this theorem for a separated scheme $X$ of finite type over the field $\FF_q$ with $q$ elements. For this, we need to introduce for any prime $\ell$ an appropriate notion of admissible $\ell$-adic Lie extensions, a replacement for the $\Lambda(G)$-module $\X(F_{\infty})$, and the corresponding $L$-functions. This will be the content of the following paragraphs. To avoid technicalities, we will additionally assume that $X$ is geomtrically connected. However, we will not exclude the case $\ell=p$, where $p$ denotes the characteristic of $\FF_q$.

 \subsection{Admissible Lie extensions}

 In the situation of Theorem~\ref{thm:mc number field}, let $\Order_F$ denote the ring of integers of $F$ and $W$ the open subscheme of $\Spec \Order_F$ given by the open complement of $\Sigma$. The $\ell$-adic Lie group $G$ is a factor group of the Galois group $\Gal(F_{\Sigma}/F)$ of the maximal extension $F_{\Sigma}$ of $F$ unramified outside $\Sigma$ within a fixed algebraic closure $\bar{F}$ of $F$. The Galois group $\Gal(F_{\Sigma}/F)$ is in geometric terms the \emph{\'etale fundamental group} $\fundG(W,\xi)$ of the scheme $W$ with respect to the geometric point
 $$
 \xi\colon \Spec \bar{F}\mto W
 $$
 corresponding to the algebraic closure $\bar{F}$.

 The \'etale fundamental group $\fundG(X,\xi)$ with respect to a fixed geometric base point $\xi$ of $X$ is in fact defined for any connected scheme. The open normal subgroups $U\onsg\fundG(X,\xi)$ are in $1$-$1$-correspondence with the isomorphism classes of finite connected pointed Galois coverings
 $$
 f\colon (Y,\eta)\mto(X,\xi)
 $$
 with Galois group
 $$
 \Gal(Y/X)=\Aut_X(Y)=\fundG(X,\xi)/U
 $$
 \cite[Ch.~1,\S 5]{Milne:EtCohom}. In the following, we will allow ourselves to neglect the base points in our notation.

 We extend the above correspondence to closed normal subgroups $V$ of $\fundG(X,\xi)$ by writing
 $$
 f\colon Y\mto X
 $$
 for the projective system of $X$-schemes $(f_U\colon Y_U\mto X)$ where $U$ runs over the open normal subgroups of $\fundG(X,\xi)/V$ and $Y_U$ denotes the Galois cover associated to the preimage of $U$ in the fundamental group of $X$. The pro-scheme $Y$ will then be called a \emph{Galois extension} of $X$ with \emph{Galois group}
 $$
 \Gal(Y/X)=\fundG(X,\xi)/V
 $$

 There exists a precise analogue of the cyclotomic $\Z_\ell$-extension of a number field in the geometric setting, namely the unique $\Z_\ell$-extension $\FF_{q^{\ell^{\infty}}}/\FF_q$ of the base field. If $\ell\neq p$ and if we suppose that $X$ is connected and normal, this is in fact the only $\Z_{\ell}$-extension by the Katz-Lang finiteness theorem \cite[Thm. 2.8]{KerzSchmidt}. There are examples of non-normal $X$ with additional $\Z_{\ell}$-extensions. For $\ell=p$ one can find in general infinitely many other $\Z_p$-extensions of $X$.

 The analogue of an admissible $\ell$-adic Lie extension of $F$ is then given by the following

 \begin{definition}
 An Galois extension $Y\mto X$ is defined to be an \emph{admissible $\ell$-adic Lie extension} if the Galois group $\Gal(Y/X)$ is an $\ell$-adic Lie group and if $Y\mto X$ factors through $X\times_{\FF_q}\FF_{q^{\ell^{\infty}}}$.
 \end{definition}

 Comparing with the definition in the number field case \cite[Sect. 1]{CoatesKim:Introduction}, we note that the extension $Y\mto X$ is by definition (pro-)\'etale and therefore unramified. Of course, there is no analogue of a totally real extension for our scheme $X$ and we may simply drop this extra condition.

 \subsection{Algebraic K-Theory}

We chose a slightly different approach to algebraic $K$-theory than the one given in \cite[\S 1]{Sujatha:Reductions}.

Let $R$ be any ring and $S\subset R$ a \emph{left denominator set}, i.\,e.\ a multiplicatively closed subset satisfying
\begin{enumerate}
\item (\emph{Ore condition}) for each $s\in S$, $b\in R$ there exist $s'\in S$, $b'\in R$ such that $b's=s'b$, and
\item (\emph{annihilator condition}) for each $s\in S$, $b\in R$ with $bs=0$ there exists $s'\in S$ with $s'b=0$.
\end{enumerate}
Then the left quotient ring $R_S$ with respect to $S$ exists and one can define a long exact localisation sequence
$$
K_1(R)\mto K_1(R_S)\xrightarrow{\del}K_0(R,R_S)
$$
\cite{WeibYao:Localization}. If $R$ is left noetherian, then the annihilator condition is implied by the Ore condition. A left denominator set is then also referred to as \emph{left Ore set}.

We can obtain the following explicit description of the sequence in terms of perfect complexes. Recall that a complex of modules over a ring $R$ is said to be \emph{strictly perfect} if it is bounded and all modules are finitely generated and projective. The complex is called \emph{perfect} if it is quasi-isomorphic to a strictly perfect complex.

\begin{theorem}\label{thm:rep of K_1}
Let $R$ be any ring and $S\subset R$ be a left denominator set.
\begin{enumerate}
 \item\label{item:fristDescription}
 The group $K_1(R)$ is as abelian group generated by symbols $[f]$ where $f$ is a quasi-automorphism of a perfect complex of $R$-modules $\cmplx{P}$. Moreover, the following (possibly incomplete) list of relations is satisfied:
\begin{enumerate}
\item $[g]+[h]=[g\comp h]$  if $g$ and $h$ are quasi-automorphisms of the same complex $\cmplx{P}$.
\item $[f]=[f']$ if there exists a quasi-isomorphism $a$ such that the diagram
$$
\xymatrix{
\cmplx{P}\ar[r]^f\ar[d]^a&\cmplx{P}\ar[d]^a\\
\cmplx{Q}\ar[r]^{f'}&\cmplx{Q}
}
$$
commutes in the derived category of complexes of $R$-modules.
\item $[b]=[a]+[c]$ if there exists an exact sequence  $0\mto \cmplx{A}\mto\cmplx{B}\mto\cmplx{C}\mto 0$ of perfect complexes such that the diagram
$$
\xymatrix{
0\ar[r]&\cmplx{A}\ar[r]\ar[d]^a&\cmplx{B}\ar[r]\ar[d]^b&\cmplx{C}\ar[r]\ar[d]^c&0\\
0\ar[r]&\cmplx{A}\ar[r]&\cmplx{B}\ar[r]&\cmplx{C}\ar[r]&0
}
$$
commutes in the non-derived category of complexes.
\end{enumerate}
\item The group $K_1(R_S)$ is generated by symbols $[f]$ where $f$ is a morphism of a perfect complex of $R$-modules $\cmplx{P}$ such that the localisation $f_S$ is a quasi-automorphism. The symbols $[f]$ satisfy the same relations as the symbols $[f_S]$ in description~(\ref{item:fristDescription}).
\item The group $K_0(R,R_S)$ is presented by generators $[\cmplx{P}]$ for each perfect complex $\cmplx{P}$ of $R$-modules such that the localisation $\cmplx{P}_S$ has vanishing cohomology and the relations
\begin{enumerate}
\item $[\cmplx{P}]=[\cmplx{Q}]$ if $\cmplx{P}$ and $\cmplx{Q}$ are quasi-isomorphic.
\item $[\cmplx{B}]=[\cmplx{A}]+[\cmplx{C}]$ if there exists a distinguished triangle
$$\cmplx{A}\mto\cmplx{B}\mto\cmplx{C}\mto\cmplx{A[1]}$$ in the derived category of $R$-modules.
\end{enumerate}
\item $\del\colon K_1(R_S)\mto K_0(R,R_S)$ is given by $\del[f]=-[\cmplx{\Cone(f)}]$ where $\cmplx{\Cone(f)}$ denotes the cone of $f$.
\end{enumerate}
In the above description, one may replace every occurrence of \emph{perfect complex} by \emph{strictly perfect complex}.
\end{theorem}
\begin{proof}
This follows from the construction of the localisation sequence from a cofibre sequence of certain Waldhausen categories \cite{WeibYao:Localization}, the algebraic description of the first Postnikov section of the associated topological spaces \cite{MT:1TWKTS}, the identicalness of $K$-theory and 'derived' $K$-theory in degrees $0$ and $1$ \cite[Thm.~5.1]{Muro:Maltsiniotis}, and the explicit description of the boundary homomorphism \cite[Thm.~A.5]{Witte:MCVarFF}.
\end{proof}

\begin{remark}
We stress that Theorem~\ref{thm:rep of K_1} might not give a presentation of $K_1(R)$ or $K_1(R_S)$. We have only listed the relations which are obvious from the description as the kernel of a certain group homomorphism given in \cite{MT:1TWKTS}. The precise set of relations will not be needed in the subsequent arguments.
\end{remark}

The derived tensor product with an $R'$-$R$-bimodule $M$ which is finitely generated and projective as $R'$-module defines a group homomorphism $K_1(R)\mto K_1(R')$. In order to define this map in terms of the above presentation with perfect complexes, we can use the following

\begin{lemma}\label{lem:replacing quasiautomorphisms}
Let $f\colon \cmplx{P}\mto\cmplx{P}$ be a endomorphism of a perfect complex $\cmplx{P}$ of $R$-modules. There exists a endomorphism $f'$ of a strictly perfect complex $\cmplx{Q}$ and a quasi-isomorphism $q$ such that the diagram
$$
\xymatrix{
\cmplx{Q}\ar[r]^{q}\ar[d]^{f'}&\cmplx{P}\ar[d]^f\\
\cmplx{Q}\ar[r]^{q}&\cmplx{P}
}
$$
commutes in the derived category of complexes of $R$-modules.
\end{lemma}
\begin{proof}
This follows from the well-known fact that the morphisms from a strictly perfect complex $\cmplx{Q}$ to a complex $\cmplx{P}$ in the derived category are the same as the set of homotopy classes of complex homomorphisms $\cmplx{Q}\mto\cmplx{P}$.
\end{proof}

We may now define $K_1(R)\mto K_1(R')$ by mapping a generator $[f]$ to $[M\tensor_R f']$ for any choice of $f'$ as in the lemma. The construction extends to our presentation of $K_1(R_S)$ and $K_0(R,R_S)$ in the obvious way.

Given an admissible $\ell$-adic Lie extension $Y\mto X$ we shall put as in the number field case
 $$
 G=\Gal(Y/X),\qquad \Gamma=\Gal(\FF_{q^{\ell^{\infty}}}/\FF_q),\qquad H=\ker(G\mto \Gamma).
 $$
 In particular, $G$ is a semidirect product of $H$ and $\Gamma$. We will write $\Lambda_{\Order}(G)$ for its Iwasawa algebra with coefficients in the valuation ring $\Order$ of a finite field extension of $\Q_{\ell}$ and $\Lambda_{\Order}(G)_{S}$ for the quotient ring with respect to \emph{Venjakob's canonical Ore set}
 $$
 S=\{f\in \Lambda(G)\colon \text{$\Lambda_{\Order}(G)/\Lambda_{\Order}(G)f$ is finitely generated as $\Lambda(H)$-module}\}.
 $$
In most situations, we omit the subscript $\Order$.

Considering the localisation sequence for $S\subset \Lambda(G)$, we can even prove that it gives rise to a split exact sequence
$$
0\mto K_1(\Lambda(G))\mto K_1(\Lambda(G)_S)\xrightarrow{\del} K_0(\Lambda(G),\Lambda(G)_S)\mto 0.
$$
\cite[Cor. 3.4]{Witte:Splitting}. We also recall that
$$
K_1(\Lambda(G))=\varprojlim_{n,U\onsg G}K_1(\Order/\ell^n\Order[G/U])
$$
carries a natural structure of a profinite group \cite[Prop. 1.5.3]{FK:CNCIT}. As in \cite{Kakde2}, we define
\begin{align*}
SK_1(\Lambda(G))&=\varprojlim_{U\onsg G}SK_1(\Order[G/U]),\\
K'_1(\Lambda(G))&=K_1(\Lambda(G))/SK_1(\Lambda(G)),\\
K'_1(\Lambda(G)_S)&=K_1(\Lambda(G)_S)/SK_1(\Lambda(G)).
\end{align*}

If $\rho\colon G\mto \GL_n(\Order)$ is a continuous character for the valuation ring $\Order$ of any finite field extension of $\Q_{\ell}$, we obtain an evaluation map
$$
\Phi_\rho\colon K_1(\Lambda(G)_S)\xrightarrow{\mathrm{tw}_{\rho}} K_1(\Lambda(G)_S)\xrightarrow{\mathrm{pr}}K_1(\Lambda(\Gamma)_S)\xrightarrow[\isomorph]{\det} \Lambda(\Gamma)_S^{\times}.
$$
Here, $\Lambda(\Gamma)_S$ is the localisation of $\Lambda(\Gamma)$ at the prime ideal generated by the maximal ideal of $\Order$, the map $\mathrm{pr}$ is the canonical projection, and the map $\mathrm{tw}_{\rho}$ is given by
$$
\mathrm{tw}_{\rho}[\cmplx{P}\xrightarrow{f}\cmplx{P}]=[\rho\tensor_{\Order}\cmplx{P}\xrightarrow{\id\tensor f}\rho\tensor_{\Order}\cmplx{P}]
$$
on morphisms $f$ of strictly perfect complexes $\cmplx{P}$. If $\rho\colon G\mto \GL_n(\Order)$ has finite image, it is immediate that $\Phi_\rho(SK_1(\Lambda(G))=1$ and hence, $\Phi_\rho$ factors through $K'_1(\Lambda(G)_S)$.

We stress that none of the results of this section are specific to the geometric nature of our conjecture. They apply equally well to the number field case.

 \subsection{A Crash Course in Etale Cohomology}

 Next, we need to find an analogue of the $\Lambda(G)$-module $\X(F_{\infty})$ that features in the main conjecture for number fields. For this, it is necessary to take a step back and have a look at the larger picture.

 In the following, we will make heavy use of the formalism of \'etale cohomology. We have to refer the reader to \cite{Milne:EtCohom} for a thorough introduction. \'Etale cohomology is a cohomology theory in the spirit of sheaf cohomology on topological spaces. It is related to the cohomology of the \'etale fundamental group just as this sheaf cohomology is related to the cohomology of the classical fundamental group.

 For any noetherian ring $R$ and any scheme $X$ (of finite type over a sufficiently nice base scheme, e.\,g.\ $\Spec \FF_q$ or $\Spec \Z$) one defines a certain abelian category of \emph{constructible \'etale sheaves of $R$-modules} $\cat{Sh}(X,R)$ together with
\begin{itemize}
\item a global section functor $\Sect_{\et}(X,\cdot)$ assigning an $R$-module to any sheaf in $\cat{Sh}(X,R)$,
\item constructions of inverse image functors $f^*\colon\cat{Sh}(X',R)\mto\cat{Sh}(X,R)$ and direct image functors $f_*\colon\cat{Sh}(X,R)\mto\cat{Sh}(X',R)$ for morphisms $f\colon X\mto X'$,
\item an extension-by-zero functor $j_!\colon \cat{Sh}(X,R)\mto\cat{Sh}(X',R)$ for open immersions $j\colon X\mto X'$,
\item a tensor product $\sheaf{F}\tensor_{R}\sheaf{G}$ of sheaves $\sheaf{F}$ and $\sheaf{G}$ in $\cat{Sh}(X,R)$,
\item for any sheaf $\sheaf{F}$ on $X$ a Godement resolution $\cmplx{\God_X(\sheaf{F})}$, i.\,e.\ a complex  of flasque sheaves that may be used to define higher derived functors.
\end{itemize}
We may apply the section functor to each degree of the Godement resolution to define a total derived section functor
$$
\RDer\Sect_{\et}(X,\sheaf{F})=\Sect_{\et}(X,\cmplx{\God_X}(\sheaf{F}))
$$
and likewise, we may also define a total higher derived image functor $\RDer f_*$. If $f\colon X\mto Y$ is a separated morphism of finite type over a noetherian scheme $Y$, we can fix a commutative diagram
$$
\xymatrix{
X\ar[r]^{j}\ar[dr]^f& X^c\ar[d]^{f^c}\\
&Y}
$$
with an open immersion $j$ and a proper morphism $f^c$. We may then define a total derived image functor with proper support
$$
\RDer f_!\sheaf{F})=f^c_*\cmplx{\God}_{X^c}(j_!\sheaf{F})
$$
as an analogue of the total derived image functor with compact support from topology. (One can prove that this construction is independent of the choice of $X^c$ up to quasi-isomorphism, see \cite[Arcata, IV, \S 5]{SGA4h}.) If $X$ is separated and of finite type over $\FF_q$, we may apply this construction to the structure morphism $s\colon X\mto \Spec \FF_q$ to define a total derived section functor with proper support
$$
\RDer \Sectc(X,\sheaf{F})=\Sect_{\et}(\Spec \FF_q,\RDer s_!\sheaf{F}).
$$

This construction works fine for finite rings $R$, but it does not give the right cohomology groups if we apply it directly to rings such as $\Z_{\ell}$ or $\Lambda(G)$. Instead, we have go a step further and define a continuous version of it. The following definition is a straight-forward generalisation of \cite[Rapport, Def. 2.1]{SGA4h} (see also \cite[Exp. VI]{SGA5}).

\begin{definition}\label{def:flat R-sheaf}
Let $R$ be a profinite ring. A \emph{flat $R$-sheaf} on $X$ is a projective system $(\sheaf{F}_I)$ indexed by the open (two-sided) ideals of $R$ such that $\sheaf{F}_I$ is a flat sheaf in $\cat{Sh}(X,R/I)$ and such that for $I\subset J$ the transition morphism $\sheaf{F}_I\mto \sheaf{F}_J$ factorises through an isomorphism
$$
R/J\tensor_{R/I}\sheaf{F}_I\isomorph \sheaf{F}_J.
$$
We denote the category of flat $R$-sheafs by $\Sh(X,R)$.
\end{definition}

The above constructions of $\RDer f_*$ and $\RDer f_!$ extend to $R$-sheaves by applying them to each element of the projective system individually. We redefine the total derived section functor by additionally taking the total derived inverse limit of the resulting projective system of complexes.

In order to give an $R$-sheaf on $X$, it suffices to know it on a cofinal system of open ideals. In particular, if $\sheaf{F}=(\sheaf{F}_{(\ell^n)})$ is a flat $\Z_{\ell}$-sheaf and $f\colon Y\mto X$ is an admissible $\ell$-adic Lie extension with Galois group $G$, we obtain a flat $\Lambda(G)$-sheaf
$$
\sheaf{F}_G=({f_U}_*f_U^*\sheaf{F}_{(\ell^n)})_{U\onsg G,n>0}.
$$
To check the conditions in Definition~\ref{def:flat R-sheaf}, it suffices to observe that the stalk $({f_U}_*f_U^*\sheaf{F}_{(\ell^n)})_{\xi}$ in a geometric point $\xi$ of $X$ is isomorphic to $\Z/\ell^n\Z[G/U]\tensor_{\Z/\ell^n\Z}(\sheaf{F}_{\xi})$.

Moreover, we may assign to each continuous $\Z_{\ell}$-representation
$$
\rho\colon \fundG(X,\xi)\mto \GL_k(\Z_{\ell})
$$
a flat $\Z_\ell$-sheaf
$$
\sheaf{M}(\rho)=(\rho^{-1}\tensor_{\Lambda(\fundG(X,\xi))}{f_{U_n}}_*f_{U_n}^*\Z/\ell^n\Z)_{n\geq 0}
$$
by choosing $U_n=\ker(\fundG(X,\xi)\mto \GL_k(\Z/\ell^n\Z))$. In this way, the category of continuous $\Z_{\ell}$-representations becomes a full subcategory of $\Sh(X,\Z_{\ell})$. More generally, we may replace $\Z_{\ell}$ by any compact noetherian commutative ring. If $X$ is the \'etale analogue of a $K(\pi,1)$-space, e.\,g.\ a smooth affine curve over a finite field of characteristic prime to $\ell$ or an open dense subscheme of $\Spec O_{F}[\frac{1}{\ell}]$, then the \'etale cohomology $\RDer\Sect_{\et}(X,\sheaf{M}(\rho))$ agrees with the continuous group cohomology of $\rho$, but in general, it is the \'etale cohomology and not the group cohomology that leads to the right constructions.

The following complex is our replacement for the module $\X(F_{\infty})$. We will explain this in more detail in Section~\ref{sec:Function fields}.

\begin{definition}
For any flat $\Z_\ell$-sheaf $\sheaf{F}$ we set
$$
C(Y/X,\sheaf{F})=\RDer\Sectc(X,\sheaf{F}_G)
$$
\end{definition}

By construction $C(Y/X,\sheaf{F})$ is  a complex of $\Lambda(G)$-modules. We can say even more:

\begin{proposition}\label{prop:perfectness}
The complex $C(Y/X,\sheaf{F})$ is a perfect complex of $\Lambda(G)$-modules, i.\,e.\ there exists a bounded complex $\cmplx{Q}$ of finitely generated, projective $\Lambda(G)$-modules and a quasi-isomorphism $q\colon\cmplx{Q}\mto C(Y/X,\sheaf{F})$.
\end{proposition}
\begin{proof}
By \cite[p. 95, Th. 4.9]{SGA4h}, the complex $\RDer\Sectc(X,{f_U}_*f_U^*\sheaf{F})$ is a perfect complex of $\Z/\ell^n\Z[G/U]$-modules for any flat \'etale sheaf of $\Z/\ell^n\Z$-modules $\sheaf{F}$. To pass from this statement to the statement of the proposition is not completely straightforward. We give some details, following \cite[XV, p. 472--492]{SGA5} and clarifying \cite[Prop. 3.1.(ii)]{Burns:MCinGIwTh+RelConj}. A slightly different proof is given in \cite[Prop. 5.2.3 + Def. 5.4.13]{Witte:PhD}. Both proofs use that $\Lambda(G)$ is compact for the topology defined by the powers of the Jacobson radical $\Jac(\Lambda(G))$:
$$
\Lambda(G)=\varprojlim_{n}\Lambda(G)/\Jac(\Lambda(G))^n.
$$
In particular, $G$ admits a fundamental system of neighbourhoods $(U_n)$ of open normal subgroups indexed by the positive integers such that the kernel of $\Lambda(G)\mto\Z/\ell\Z[G/U_1]$ is contained in $\Jac(\Lambda(G))$. The K\"unneth formula for the trivial product of $X$ with the spectrum of its base field implies that there exists a quasi-isomorphism
\begin{gather*}
\Z/\ell^n\Z[G/U_n]\tensor^{\LDer}_{\Z/\ell^{n+1}\Z[G/U_{n+1}]}\RDer\Sectc(X,{f_{U_{n+1}}}_*f_{U_{n+1}}^*\sheaf{F}_{(\ell^{n+1})})\\
\downarrow\\
\RDer\Sectc(X,\Z/\ell^n\Z[G/U_n]\tensor_{\Z/\ell^{n+1}\Z[G/U_{n+1}]}{f_{U_{n+1}}}_*f_{U_{n+1}}^*\sheaf{F}_{(\ell^{n+1})})\\
\parallel\\
\RDer\Sectc(X,{f_{U_{n}}}_*f_{U_{n}}^*\sheaf{F}_{(\ell^{n})}).
\end{gather*}
By sucessively applying \cite[XV, 3.3, Lemma 1]{SGA5} (the proof of which uses lifting of idempotents for $\Lambda(G)$ as a central ingredient) one finds a quasi-isomorphism of projective system of complexes
$$
(q_n)\colon(\cmplx{Q}_n)\mto (\RDer\Sectc(X,{f_{U_{n}}}_*f_{U_{n}}^*\sheaf{F}_{(\ell^{n})})
$$
with $\cmplx{Q}_n$ being a strictly perfect complex of $\Z/\ell^n\Z[G/U_n]$-modules such that
$$
\Z/\ell^n\Z[G/U_n]\tensor_{\Z/\ell^{n+1}\Z[G/U_{n+1}]}\cmplx{Q}_{n+1}\mto \cmplx{Q}_n
$$
is an \emph{isomorphism} (not merely a quasi-isomorphism) of complexes. We set
$$
\cmplx{Q}=\varprojlim_{n}\cmplx{Q}_n.
$$
Using that projective systems of finite modules are $\varprojlim$-acyclic \cite[Cor. 7.2]{Jensen:FuncteusDerives}, we obtain a quasi-isomorphism
$$
q\colon \cmplx{Q}\mto C(Y/X,\sheaf{F}).
$$
Each $Q^k$ is projective as a compact $\Lambda(G)$-module \cite[Cor.~3.3]{Brumer:PseudocompactAlgebras} and the complex $\cmplx{Q}$ is bounded. Moreover, since the transition morphisms in the system $(\cmplx{Q}_n)$ are surjective, we have
$$
\Z/\ell\Z[G/U_1]\ctensor_{\Lambda(G)}\cmplx{Q}=\varprojlim_{n}\Z/\ell\Z[G/U_1]\tensor_{\Z/\ell^n\Z[G/U_n]}\cmplx{Q}_n=\cmplx{Q}_1
$$
\cite[Lemma~A.4]{Brumer:PseudocompactAlgebras}. Hence, each $Q^k$ is finitely generated by the topological Nakayama lemma. In particular, $Q^k$ is also projective as abstract $\Lambda(G)$-module.
\end{proof}

The K\"unneth-formula for the trivial product as in the proof above shows that $C(Y/X,\sheaf{F})$ behaves well with respect to derived tensor products. In particular, we have quasi-isomorphisms
$$
  \Lambda(\Gal(Y'/X))\tensor^{\LDer}_{\Lambda(\Gal(Y/X))}C(Y/X,\sheaf{F})\qiso C(Y'/X,\sheaf{F})
$$
for any subextension $Y'/X$ of $Y/X$. Moreover,
$$
\rho\tensor^{\LDer}_{\Z_{\ell}}C(Y/X,\sheaf{F})\qiso C(Y/X,\sheaf{M}(\rho)\tensor_{\Z_{\ell}}\sheaf{F})
$$
for any continuous $\ell$-adic representation $\rho$ of $\Gal(Y/X)$.

\subsection{L-Functions}

In this subsection we will recall Grothendieck's and Deligne's fundamental results on $L$-functions for $\Z_{\ell}$-sheaves on the scheme $X$ over $\FF_q$. In the case $\ell=p$ this is complemented by a result of Emmerton and Kisin \cite{EmertonKisin}.

Let $X$ be a geometrically connected scheme of finite type over $\FF_q$. We let $X^0$ denote the set of closed points of $X$. For any $x\in X^0$ we let $\deg(x)$ denote the degree of the residue field $k(x)$ of $x$ over $\FF_q$ and $N(x)=q^{\deg(x)}$ the order of $k(x)$. Furthermore, we fix an algebraic closure $\bar{k}(x)$ of $k(x)$ and let denote $\Frob_x\in\fundG(X,\xi)$ its geometric Frobenius element acting by $a\mapsto a^{N(x)^{-1}}$ for $a\in \bar{k}(x)$. For any representation
$$
\rho\colon\fundG(X,\xi)\mto \GL_n(\Order_E)
$$
on the ring of integers $\Order_E$ of a number field $E$, we may define a complex function
$$
L^A(\rho,s)=\prod_{x\in X^0}E_x(N(x)^{-s})^{-1},\qquad E_x(t)=\det(1-\rho(\Frob_x)t)
$$
as direct analogue of an Artin L-function for a number field. Different from the number field case, we obtain $L^A(\rho,s)$ by evaluating the formal power series with $\Order_E$-coefficients
$$
L(\rho,t)=\prod_{x\in X^0}E_x(t^{\deg(x)})^{-1}
$$
in $q^{-s}$. In particular, we may view $L(\rho,t)$ as an element of $\Order[[t]]^{\times}$ for any completion $\Order$ of $\Order_E$ at a prime of $E$. (We warn the reader that this is not yet the precise analogue of the Iwasawa power series of an $\ell$-adic $L$-function.) More generally, one may define $L(\sheaf{F},t)\in\Order[[t]]^{\times}$ for any compact commutative noetherian $\Z_{\ell}$-algebra $\Order$ and any flat $\Order$-sheaf $\sheaf{F}$ by taking the product over the inverses of the characteristic polynomials of the Frobenius operation on the stalks of $\sheaf{F}$.

There is a second power series that we may attach to the $\Order$-sheaf $\sheaf{F}$. Recall that the determinant induces an isomorphism
$$
\det\colon K_1(\Order[[t]])\mto \Order[[t]]^{\times}.
$$
Let $s\colon X\mto \Spec \FF_q$ denote the structure morphism. By the same argument as in the proof of Proposition~\ref{prop:perfectness}, the complex
$
\Sect_{\et}(\bar{\FF}_q,\RDer s_!\sheaf{F})
$
is a perfect complex of $\Order$-modules with an action of the Frobenius $\Frob\in \Gal(\bar{\FF}_q/\FF_q)$. We may set
$$
L_!(\sheaf{F},t)=\det[1-t\Frob\colon \Order[[t]]\tensor_{\Order}\Sect_{\et}(\bar{\FF}_q,\RDer s_!\sheaf{F})]^{-1}\in \Order[[t]]^{\times}.
$$
The interesting property of $L_!(\sheaf{F},t)$ is that it is in fact a rational function. To see this, we use Lemma~\ref{lem:replacing quasiautomorphisms} to replace $\Frob$ by an endomorphism $f$ on a strictly perfect complex $\cmplx{Q}$. By the relations given in Theorem~\ref{thm:rep of K_1} we may then write
$$
L_!(\sheaf{F},t)=\prod_{n\in\Z} \det[1-tf\colon \Order[[t]]\tensor_{\Order}Q^n]^{(-1)^{n+1}}.
$$

We introduce the power series
$$
v(\sheaf{F},t)=L(\sheaf{F},t)/L_!(\sheaf{F},t)
$$
to measure the difference. Furthermore, we set
$$
\Order\langle t\rangle=\varprojlim_{n}\Order/\Jac(\Order)^n[t]
$$
and write $\Order\{t\}$ for the localisation of $\Order\langle t\rangle$ at the multiplicatively closed subset of those elements which become a unit in $\Order[[t]]$. In particular, we have
$$
L_!(\sheaf{F},t)\in\Order\{t\}^{\times}.
$$

\begin{theorem}\label{thm:trace formula}
Let $\ell$ be any prime, $\Order$ a compact, commutative, noetherian $\Z_{\ell}$-algebra and $\sheaf{F}$ an $\Order$-sheaf on $X$.
\begin{enumerate}
\item (Grothendieck-Deligne) If $\ell\neq p$ then $v(\sheaf{F},t)=1$.
\item (Emmerton-Kisin) If $\ell=p$ then $v(\sheaf{F},t)\in\Order\langle t\rangle^{\times}$.
\end{enumerate}
In particular, we have $L(\sheaf{F},t)\in\Order\{t\}^{\times}$ in both cases.
\end{theorem}
\begin{proof}
Assertion $(1)$ follows from \cite[Fonctions $L$ mod $\ell$, Thm.~2.2]{SGA4h} by passing to the projective limit. Assertion $(2)$ for $q=p$ is \cite[Cor.~1.8]{EmertonKisin}. It remains true for $q=p^n$ because the $v$ for a scheme $X/\FF_q$ considered as a scheme over $\FF_p$ is obtained from the original $v$ by replacing $t$ by $t^n$.
\end{proof}

In fact, Theorem~\ref{thm:trace formula}.(1) remains valid for non-commutative coefficients \cite{Witte:NoncommutativeLFunctions}. The formulation of a reasonable non-commutative version of Theorem~\ref{thm:trace formula}.(2) poses additional technical difficulties related to the fact that for non-commutative $\Order$, the map $K_1(\Order\langle t \rangle)\mto K_1(\Order[[t]])$ might not be injective.

We use the above theorem to construct the true analogue of the Iwasawa power series of the classical $\ell$-adic $L$-function.

\begin{lemma}
Let $\gamma\in\Gamma\isomorph\Z_{\ell}$ be a topological generator, $\Order$ a compact, commutative, noetherian $\Z_{\ell}$-algebra. The assignment $t\mapsto \gamma^{-1}$ defines a ring homomorphism
$$
\Order\{t\}\mto\Lambda(\Gamma)_S.
$$
\end{lemma}
\begin{proof}
Clearly, $t\mapsto \gamma^{-1}$ defines a homomorphism $\Order\langle t\rangle\mto\Lambda(\Gamma)$. Let $f(t)\in \Order\langle t\rangle$ be a unit in $\Order[[t]]$. We need to show that $f(\gamma^{-1})\in S$.
The set $S$ consists of those elements in $\Lambda(\Gamma)$ whose reduction modulo every maximal ideal $m$ of $\Order$ is not zero. Now the reduction of $f(\gamma^{-1})$ modulo $m$ is a polynomial in $\gamma^{-1}$ whose constant coefficient is a unit in $\Order/m$.
\end{proof}

In particular, we may apply this lemma to $\Gamma=\Gal(\FF_{q^{\ell^{\infty}}}/\FF_q)$ with $\gamma$ being the image of the Frobenius $\Frob\in\Gal(\bar{\FF}_q/\FF_q)$. We then have elements
$$
L(\sheaf{F},\gamma^{-1}),\quad L_!(\sheaf{F},\gamma^{-1})\in \Lambda(\Gamma)_S^{\times}.
$$
for every flat $\Order$-sheaf $\sheaf{F}$ on $X$. Assume that $\ell\neq p$ and let $\kappa\colon \Gal(\bar{\FF}_q/\FF_q)\mto \Z_{\ell}^{\times}$ denote the cyclotomic character. For any Artin character $\rho\colon\fundG(X,\xi)\mto \GL_k(\Order_E)$ the elements $L(\sheaf{M}(\rho),\gamma^{-1})$ satisfy the interpolation property
$$
L(\sheaf{M}(\rho\kappa^n),1)=L^A(\rho,n).
$$
for  $n\in\Z$. (This holds also for the leading terms.)

For $\ell=p$ one should not expect to obtain an analogous result, for in the $p$-adic world, the Tate twist of the motive defined by $\rho$ does not correspond to a continuous character $\kappa$, but takes on a very different shape. However, we may still view $L(\sheaf{F},\gamma^{-1})$ as an interpolation of the leading terms of $L(\sheaf{M}(\rho)\tensor_{\Z_{\ell}}\sheaf{F},t)$ at $t=1$.

\subsection{The Main Conjecture}

We are now ready to formulate the non-commutative main conjecture of Iwasawa theory in the geometric case. For any profinite group $G$, let $\Cr(G)$ denote the set of continuous representations $\rho\colon G\mto \GL_k(\Order_{\rho})$ for the valution ring $\Order_{\rho}$ of a finite field extension of $\Q_p$. Likewise, we let $\Ar(G)$ denote the set of all Artin representations $\rho\colon G\mto \GL_k(\Order_{\rho})$, i.\,e.\ representations with finite image.

\begin{theorem}[\cite{Witte:MCVarFF}]\label{thm:mythm}
Fix any prime $\ell$. Let $f\colon Y\mto X$ be an admissible $\ell$-adic Lie extension of a geometrically connected scheme $X$ of finite type over $\FF_q$. Set $G=\Gal(Y/X)$. For any flat $\Z_{\ell}$-sheaf $\sheaf{F}$, the cohomology of $C(Y/X,\sheaf{F})$ is $S$-torsion and there exists an element $\zeta_!(\sheaf{F})\in K_1(\Lambda_{\Z_{\ell}}(G)_S)$ such that
\begin{enumerate}
\item $\del \zeta_!(\sheaf{F})=[C(Y/X,\sheaf{F})]$ in $K_0(\Lambda_{\Z_{\ell}}(G),\Lambda_{\Z_{\ell}}(G)_S)$,
\item $\Phi_\rho(\zeta_!(\sheaf{F}))=L_!(\sheaf{M}(\rho)\tensor_{\Z_{\ell}}\sheaf{F},\gamma^{-1})$ in $\Lambda_{\Order_{\rho}}(\Gamma)_S^{\times}$ for any continuous representation $\rho\in \Cr(G)$.
\end{enumerate}
\end{theorem}

Recall that for $\ell\neq p$, we have $L_!(\sheaf{F},t)=L(\sheaf{F},t)$ in the commutative situation. Thus, in this case, the above result is indeed a full analogue of Theorem~\ref{thm:mc number field}. Note that we do not need to assume a $\mu=0$-conjecture. We can directly prove the much stronger $S$-torsion property.

In fact, the result in \cite{Witte:MCVarFF} is more general than the one stated above. By replacing the fundamental group by an arbitrary principal covering, one may also deal with nonconnected schemes. Moreover, the theorem still holds if $G$ is no longer an $\ell$-adic Lie group, but topologically finitely generated and virtually pro-$\ell$. This applies for example to the maximal pro-$\ell$-quotient of the fundamental group $\fundG(X,\xi)$. We may also replace $\Z_\ell$ by more general coefficient rings, for example, the valuation ring of any finite field extension of $\Q_\ell$.

For $\ell=p$, Burns proves the following result:

\begin{theorem}[\cite{Burns:MCinGIwTh+RelConj}]\label{thm:Burnsthm}
Let $p\neq 2$. Let $f\colon Y\mto X$ be an admissible $p$-adic Lie extension of a geometrically connected scheme $X$ of finite type over $\FF_q$. Set $G=\Gal(Y/X)$. For any flat $\Z_{p}$-sheaf $\sheaf{F}$, there exists an element $\zeta(\sheaf{F})\in K_1(\Lambda_{\Z_p}(G)_S)$ such that
\begin{enumerate}
\item $\del \zeta(\sheaf{F})=[C(Y/X,\sheaf{F})]$ in $K_0(\Lambda_{\Z_p}(G),\Lambda_{\Z_p}(G)_S)$,
\item $\Phi_\rho(\zeta(\sheaf{F}))=L(\sheaf{M}(\rho)\tensor_{\Z_{p}}\sheaf{F},\gamma^{-1})$ in $\Lambda_{\Order_{\rho}}(\Gamma)_S^{\times}$ for any Artin representation $\rho\in \Ar(G)$.
\end{enumerate}
\end{theorem}

Again, the restriction to geometrically connected schemes and to $\ell$-adic Lie groups may easily be lifted. However, it is not clear that $\zeta(\sheaf{F})$ satisfies the interpolation property with respect to all continuous representations.  The exclusion of the prime $2$ in the statement is a purely technical restriction. There is no reason to expect any odd behaviour for the even prime. Also for technical reasons, we may not replace $\Z_p$ by the valuation ring of any finite extension of $\Q_p$. At present, we can only deal with unramified extensions.

\section{Sketch of Proofs}\label{sec:Proofs}

We will give a fairly detailed sketch of the proofs of Theorem~\ref{thm:mythm} and Theorem~\ref{thm:Burnsthm} below. In some details we will deviate from the original proofs in \cite{Witte:MCVarFF} and \cite{Burns:MCinGIwTh+RelConj}, but the general ideas remain the same.

\subsection{On the Proof of Theorem~\ref{thm:mythm}}

We begin by assuming that we already know that $C(Y/X,\sheaf{F})$ is $S$-torsion. The proof of the existence of the non-commu\-ta\-tive $L$-function $\zeta_!(\sheaf{F})$ is very different from the number field case: We are in the happy position to be able to give an  explicit construction of a hot candidate for $\zeta_!(\sheaf{F})$. The proof then boils down to verifying that this candidate does indeed satisfy the required interpolation property.

Let $\overline{X}$ be the base change of $X$ to the  algebraic closure $\bar{\FF}_q$ of $\FF_q$ and set
$$
\overline{C}(Y/X,\sheaf{F})=\RDer\Sectc(\overline{X},\sheaf{F}_G).
$$
The geometric Frobenius $\Frob\in \Gal(\bar{\FF}_q/\FF_q)$ acts on this complex and from the Hochschild-Serre spectral sequence it follows that there exists a distinguished triangle
$$
\overline{C}(Y/X,\sheaf{F})\xrightarrow{\id-\Frob}\overline{C}(Y/X,\sheaf{F})\mto C(Y/X,\sheaf{F}).
$$
Since $C(Y/X,\sheaf{F})$ is $S$-torsion we see by Theorem~\ref{thm:rep of K_1} that $[\id-\Frob]$ is a class in $K_1(\Lambda(G)_S)$. We take the inverse of it as our definition of $\zeta_!(\sheaf{F})$. Since $C(Y/X,\sheaf{F})$ is the cone of $\id-\Frob$ it follows again by Theorem~\ref{thm:rep of K_1} that $\del\zeta_!(\sheaf{F})=[C(Y/X,\sheaf{F})]$.

The construction of $[\id-\Frob]$ is compatible with taking derived tensor products. In particular, we have
$$
\Phi_{\rho}(\zeta_!(\sheaf{F}))=\zeta_!(\sheaf{M}(\rho)\tensor_{\Z_{\ell}}\sheaf{F})\in K_1(\Lambda(\Gamma)_S)
$$
for any continuous representation $\rho\in\Cr(G)$. It remains to verify that $\zeta_!(\sheaf{F})=L_!(\sheaf{F},\gamma^{-1})$ for any flat $\Z_{\ell}$-sheaf $\sheaf{F}$ and the cyclotomic $\Z_{\ell}$-extension $X^{\cyc}/X$. This follows from the commutativity of the diagram
\begin{equation}\label{eqn:the diagram}
\xymatrix{
\Lambda(\Gamma)\tensor_{\Z_{\ell}}^{\LDer}\RDer\Sectc(\overline{X},\sheaf{F})\ar[r]^{\gamma^{-1}\tensor \Frob}\ar[d]^{\eta}&
\Lambda(\Gamma)\tensor_{\Z_{\ell}}^{\LDer}\RDer\Sectc(\overline{X},\sheaf{F})\ar[d]^{\eta}\\
\RDer\Sectc(\overline{X},\sheaf{F}_{\Gamma})\ar[r]^{\Frob}&\RDer\Sectc(\overline{X},\sheaf{F}_{\Gamma})
}
\end{equation}
in the derived category of complexes of $\Lambda(\Gamma)$-modules, with $\eta$ denoting the canonical quasi-isomorphism.

We will now give a sketch of the proof of the $S$-torsion property. As before, let $H=\ker(G\mto \Gamma)$. The general case is easily reduced to the case that $H$ is finite by considering quotients by open subgroups of $H$ which are normal in $G$. From now on, we assume that $H$ is finite. In particular, $Y$ may be viewed as a scheme of finite type over $\FF_{q^{\ell^{\infty}}}$. Moreover, in this situation, the cohomology groups of $C(Y/X,\sheaf{F})$ are $S$-torsion precisely if they are finitely generated as $\Z_{\ell}$-modules. Hence, the $S$-torsion property follows from

\begin{proposition}
Assume that $H$ is finite. For each integer $i$,
$$
\HF^{i+1}(C(Y/X,\sheaf{F}))\isomorph\HF_{\mathrm{c}}^{i}(Y,\sheaf{F})\isomorph
\HF_{\mathrm{c}}^{i}(Y\times_{\FF_{q^{\ell^{\infty}}}}\bar{\FF}_q,\sheaf{F})^{\Gal(\bar{\FF}_q/\FF_{q^{\ell^{\infty}}})}.
$$
In particular, the cohomology groups of $C(Y/X,\sheaf{F})$ are finitely generated as $\Z_{\ell}$-modules.
\end{proposition}
\begin{proof}
We first note that
$$
\HF^i(C(Y/X,\sheaf{F}))\isomorph\varprojlim_{V}\HF_{\mathrm{c}}^i(V,\sheaf{F})
$$
where the limit goes over the finite Galois subextensions of $Y/X$. From this, we see that we may replace $X$ by an appropriate $V$ and enlarge $\FF_q$ if necessary such that we may assume that $Y$ is the cyclotomic $\Z_{\ell}$-extension of $X$ and that
$$
Y\times_{\FF_{q^{\ell^{\infty}}}}\bar{\FF}_q=\overline{X}.
$$
The cohomology groups $\HF_{\mathrm{c}}^{i}(\overline{X},\sheaf{F})$ are known to be finitely generated $\Z_{\ell}$-modules, even in the case $\ell=p$ \cite[p.~84 \S~2.10]{SGA4h} (or use Proposition~\ref{prop:perfectness} for $G=1$). Since $\Gal(\bar{\FF}_q/\FF_{q^{\ell^{\infty}}})$ is of order prime to $\ell$ the Hochschild-Serre spectral sequence gives us
$$
\HF_{\mathrm{c}}^{i}(\overline{X},\sheaf{F})^{\Gal(\bar{\FF}_q/\FF_{q^{\ell^{\infty}}})}=\HF_{\mathrm{c}}^i(Y,\sheaf{F}),
$$
which is still finitely generated over $\Z_{\ell}$.

Let $X_n$ denote the subextension of degree $\ell^n$ in $Y/X$. From the Hochschild-Serre spectral sequence we obtain a short exact sequence
$$
0\mto \HF_{\mathrm{c}}^i(Y,\sheaf{F})_{\Gamma^{\ell^n}}\mto \HF_{\mathrm{c}}^{i+1}(X_n,\sheaf{F})\mto
\HF_{\mathrm{c}}^{i+1}(Y,\sheaf{F})^{\Gamma^{\ell^n}}\mto 0
$$
As $\HF_{\mathrm{c}}^{i+1}(Y,\sheaf{F})$ is finitely generated over the noetherian ring $\Z_{\ell}$, the increasing family of submodules $\HF_{\mathrm{c}}^{i+1}(Y,\sheaf{F})^{\Gamma^{\ell^n}}$ becomes stationary. This means in turn that the inverse limit with respect to the norm maps vanishes. Hence,
$$
\HF^{i+1}(C(Y/X,\sheaf{F}))\isomorph \varprojlim_{n}\HF_{\mathrm{c}}^i(Y,\sheaf{F})_{\Gamma^{\ell^n}}
$$
is the compactification of the $\Lambda(\Gamma)$-module $\HF_{\mathrm{c}}^i(Y,\sheaf{F})$. However, this module, being finitely generated over $\Z_\ell$, is already compact and therefore,
$$
\varprojlim_{n}\HF_{\mathrm{c}}^i(Y,\sheaf{F})_{\Gamma^{\ell^n}}=\HF_{\mathrm{c}}^i(Y,\sheaf{F}).
$$
\end{proof}

\subsection{On the Proof of Theorem~\ref{thm:Burnsthm}}

In the view of Theorem~\ref{thm:mythm} we are reduced to showing

\begin{proposition}\label{prop:reformulation}
Let $Y/X$ be an admissible $p$-adic Lie extension with Galois group $G$. For any flat $\Z_{p}$-sheaf $\sheaf{F}$ on $X$, there exists a unique $\nu(\sheaf{F})\in K_1'(\Lambda_{\Z_p}(G))$ such that
$$
\Phi_{\rho}(\nu(\sheaf{F}))=v(\sheaf{M}(\rho)\tensor_{\Z_p}\sheaf{F},\gamma^{-1})
$$
in $\Lambda_{\Order_{\rho}}(\Gamma)^{\times}$ for every Artin representation $\rho$ of $G$.
\end{proposition}

We extend the defintion of $v(\sheaf{F},t)$ to the non-commutative world by setting
\begin{multline*}
K_1(A[[t]])\ni v(\sheaf{F},t)=\\
[\id-\Frob t\colon A[[t]]\tensor_{A}\RDer\Sectc(\overline{X},\sheaf{F})]-\sum_{x\in X^0}[\id-\Frob_xt^{\deg x}\colon A[[t]]\tensor_{A}\sheaf{F}_x]
\end{multline*}
for any compact noetherian ring $A$ (not necessarily commutative) and any flat $A$-sheaf $\sheaf{F}$. Note that the sum converges in the profinite topology of $K_1(A[[t]])$ because there are only finitely many closed points of $X$ of a given degree. Note further that the definition of $v(\sheaf{F},t)$ is compatible with derived tensor products in the following sense: If $B$ is another compact noetherian ring and $P$ a $B$-$A$-bimodule which is finitely generated and projective as $B$-module, then the image of $v(\sheaf{F},t)$ under the map
$$
K_1(A[[t]])\xrightarrow{P\tensor^{L}_{A}\cdot} K_1(B[[t]])
$$
is $v(P\tensor_{A}\sheaf{F},t)$.

If $A$ is commutative, then
$$
v(\sheaf{F},t)\in A\langle t\rangle^{\times}\subset A[[ t]]^{\times}=K_1(A[[t]])
$$
by Proposition~\ref{thm:trace formula}. In particular, we may consider its image $v(\sheaf{F},1)\in A^{\times}$ under the homomorphism $A\langle t\rangle\xrightarrow{t\mapsto 1} A$.

\begin{proposition}\label{prop:abelian case}
Assume that $Y/X$ is an admissible $p$-adic Lie extension with $G=\Gal(Y/X)$ abelian. For any compact commutative noetherian $\Z_p$-algebra $\Order$ and any flat $\Order$-sheaf $\sheaf{F}$, the element $\nu(\sheaf{F})=v(\sheaf{F}_G,1)$ validates the interpolation property of Proposition~\ref{prop:reformulation}.
\end{proposition}
\begin{proof}
Choose $A=\Lambda(G)$, $B=\Lambda_{\Order_{\rho}}(\Gamma)$. By the compatibility with derived tensor products we conclude that
$$
\Phi_{\rho}(\nu(\sheaf{F}))=\nu(\sheaf{M}(\rho)\tensor_{\Z_p}\sheaf{F})=v((\sheaf{M}(\rho)\tensor_{\Z_p}\sheaf{F})_{\Gamma},1).
$$
One then checks using the diagram~(\ref{eqn:the diagram}) on page~\pageref{eqn:the diagram} that
$$
v((\sheaf{M}(\rho)\tensor_{\Z_p}\sheaf{F})_{\Gamma},t)=v(\sheaf{M}(\rho)\tensor_{\Z_p}\sheaf{F},\gamma^{-1}t).
$$
\end{proof}

Unfortunately, it is a priori not clear that we can evaluate the element $v(\sheaf{F}_G,t)$ in $1$ if $G$ is not commutative. Even if we knew that $v(\sheaf{F}_G,t)$ was in the image of $K_1(\Lambda(G)\langle t\rangle)\mto K_1(\Lambda(G)[[t]])$, the map could still be so far from injective that the evaluation of an element in the preimage of $v(\sheaf{F}_G,t)$ in $t=1$ depends on the particular choice of it.

However, we will prove below that the evaluation is possible for a $p$-adic Lie group $G$ under the hypotheses
\begin{enumerate}
\item[(H1)] $G$ contains an open central subgroup $Z$,
\item[(H2)] $G$ is a pro-$p$-group,
\item[(H3)] There exists a system of representatives $R\subset G$ for the cosets in $G/Z$ which contains $1$ and consists of full $G$-orbits,
\item[(H4)] $\Order$ is the valuation ring of a finite unramified extension of $\Q_p$ and $p\neq 2$
\end{enumerate}
introduced in \cite{SchneiderVenjakob:Keins}.

Recall the notation
$$
\Sg(G,Z)=\{U\colon Z\subset U\subset G\}
$$
and the homomorphism
$$
\theta^G_Z\colon K_1(\Lambda(G))\mto
\prod_{U\in \Sg(G,Z)}\Lambda(U^{\ab})^{\times}
$$
from \cite[\S 4]{SchneiderVenjakob:Keins}. We also recall from \emph{loc. cit.}, Theorem~4.10, that under the above hypotheses the kernel of $\theta^G_Z$ is the group $SK_1(\Lambda(G))$ and that its image consists of precisely those elements satisfying the congruence conditions $(M1)$--$(M4)$ given in \emph{loc. cit.}, before Lemma~4.6.

Note that if $G$ satisfies $(H1)$--$(H3)$ then so does $G\times\Z_p$, with $Z$ replaced by $Z\times\Z_p$. Moreover, we may canonically identify
$$
\Sg(G\times \Z_p,Z\times \Z_p)\isomorph\Sg(G,Z)
$$
via the projection map. Using the isomorphism $\Lambda(G\times \Z_p)\isomorph \Lambda(G)[[t]]$ that maps $1\in \Z_p$ to the power series $1-t$ and the map $\theta^{G\times \Z_p}_{Z\times \Z_p}$ we obtain a homomorphism
$$
\theta_t\colon K_1(\Lambda(G)[[t]])\mto \prod_{U\in\Sg(G,Z)}\Lambda(U^\ab)[[t]]^{\times}.
$$
We let $K^{\diamond}_1(\Lambda(G)[[t]])$ denote the preimage of
$$
\prod_{U\in \Sg(G,Z)}\Lambda(U^{\ab})\langle t\rangle^{\times}
$$
under the homomorphism $\theta_t$.

\begin{proposition}\label{prop:evaluation}
Assume $(H1)$--$(H4)$. There exists a unique homomorphism $\varepsilon$ filling the commutative diagram
$$
\xymatrix{
K^{\diamond}_1(\Lambda(G)[[t]])\ar[r]^{\theta_t}\ar[d]^{\varepsilon}&\prod\limits_{U\in \Sg(G,Z)}\left(\Lambda(U^{\ab})\langle t\rangle \right)^{\times}\ar[d]^{t\mapsto 1}\\
K_1'(\Lambda(G))\ar[r]^{\theta^{G}_{Z}}&\prod\limits_{U\in \Sg(G,Z)}\Lambda(U^{\ab})^{\times}.
}
$$
\end{proposition}
\begin{proof}
Since the homomorphism $\theta^G_{Z}$ in the diagram is injective, it suffices to show that the evaluation in $t=1$ of an element in the image of $\theta_t$ lies in the image of $\theta^{G}_{Z}$. As noted above, the images of $\theta_t$ and $\theta^{G}_{Z}$ may be described by a list of explicit congruences $(M1)$--$(M4)$. Let
$$f(t)=(f_U(t))_{U\in \Sg(G,Z)}\in \theta_t(K^{\diamond}_1(\Lambda(G)[[t]])).$$
We will exemplarily check that $f(1)$ satisfies
$$
(M3)\qquad\operatorname{ver}^{V}_{U}(f_V(1))-f_U(1)=\sigma_U^V(x)\quad\text{for a $x\in \Lambda(U^{\mathrm{ab}})$ if $[V:U]=p$}.
$$
Here,
$$
\sigma_U^V\colon \Lambda(U^{\mathrm{ab}})\rightarrow \Lambda(U^{\mathrm{ab}}), \qquad x\mapsto \sum_{g\in V/U}gxg^{-1},
$$
and
$$
\operatorname{ver}^{V}_{U}\colon \Lambda(V^{\mathrm{ab}})\rightarrow\Lambda(U^{\mathrm{ab}})
$$
is the unique continuous ring homomorphism which coincides with the transfer map on $V^{\ab}$ and with the absolute Frobenius automorphism on the coefficient ring $\mathcal{\Order}$.

Fix $U\subset V\in \Sg(G,Z)$ with $[V:U]=p$ and let $(g,a)\in V^{\mathrm{ab}}\times\Z_p$. Then
$$\operatorname{ver}^{V\times\Z_p}_{U\times\Z_p}=(\operatorname{ver}^{V}_{U}(g),[V:U]a),$$
in particular,
$$
\operatorname{ver}^{V\times\Z_p}_{U\times\Z_p}(t)=1-(1-t)^p
$$
for the indeterminate $t$.
We conclude that the power series $\operatorname{ver}^{V\times\Z_p}_{U\times\Z_p}(f(t))$ lies in $\left(\Lambda(U^{\ab})\langle t\rangle\right)^{\times}$ and that its evaluation in $1$ agrees with $\operatorname{ver}^{V}_{U}(f_V(1))$. By assumption, there exist a power series
$$
x(t)=\sum_{i=0}^{\infty} x_i t^i\in \Lambda(U^{\mathrm{ab}})[[t]]
$$
with
$$
\operatorname{ver}^{V\times \Z_p}_{U\times \Z_p}(f_V(t))-f_U(t)=\sigma_{U\times\Z_p}^{V\times\Z_p}(x(t))=\sum_{g\in V/U}gx(t)g^{-1}\in \Lambda(U^{\ab})\langle t\rangle.
$$
Hence, for every $n$, $\sigma_U^V(x_i)\in \Jac(\Lambda(U^\mathrm{ab}))^n$ for almost all $i$. Since $\Lambda(U^{\ab})$ is compact, the image of the continuous map $\sigma_U^V$ is closed and therefore, there exists an $x\in \Lambda(U^{\mathrm{ab}})$ with
$$
\sigma_U^V(x)=\sum_{i=0}^\infty \sigma_U^V(x_i)=\operatorname{ver}^{V}_{U}(f_V(1))-f_U(1).
$$
Similarily, one shows that $f(1)$ also satisfies $(M1)$, $(M2)$, and $(M4)$. In each case, the main step is to verify that each of the maps involved in the formulation of the respective congruence relations are compatible with the evaluation in $1$. This can be achieved by a quick inspection of the corresponding definitions given in \cite{SchneiderVenjakob:Keins}.
\end{proof}

\begin{proof}[Proof of Proposition~\ref{prop:reformulation}]
Using the same reduction arguments as in \cite[\S 3]{Sujatha:Reductions} and in \cite[Prop.~2.2]{Kakde:Congruences} we see that it suffices to consider the cases
\begin{enumerate}
 \item $G$ is abelian,
 \item $H$ is a finite $p$-group,
\end{enumerate}
The first case has been settled in Proposition~\ref{prop:abelian case}. Therefore, assume that $H$ is a finite $p$-group. Using Proposition~\ref{prop:evaluation}, we may set
$$
\nu(\sheaf{F})=\varepsilon(v(\sheaf{F}_G,t)).
$$
The interpolation property is then verified as in Proposition~\ref{prop:abelian case}. This concludes the proof of Proposition~\ref{prop:reformulation} and hence, also the proof of Theorem~\ref{thm:Burnsthm}.
\end{proof}

Note that our $\theta_t$ differs from the one used in \cite{Burns:MCinGIwTh+RelConj}. Instead of using the identification $\Lambda(G)[[t]]\isomorph \Lambda(G\times\Z_p)$ Burns considers $\Lambda(G)[[t]]$ as an Iwasawa algebra with coefficient ring $\Order[[t]]$ and transfers the logarithm techniques of Oliver and Taylor to power series rings. This also yields slightly different congruence relations, but the difference vanishes if one evaluates in $t=1$.

We also remark that
$$
\ker \theta_t= SK_1(\Lambda(G\times\Z_p))= SK_1(\Lambda(G))
$$
by \cite[Prop. 5.3]{Witte:NoncommutativeLFunctions}.

\section{Main Conjectures for Function Fields}\label{sec:Function fields}

In this section, we will consider the special case that $X$ is a geometrically connected smooth affine curve over $\FF_q$. Let $F$ denote the function field of $X$, $X^c$ the smooth compact curve with function field $F$ and $\Sigma=X^c\setminus X$. Every admissible $\ell$-adic Lie extension $Y/X$ may then be seen as a Galois extension $F_{\infty}/F$ such that
\begin{enumerate}
\item the Galois group $G$ of $F_{\infty}/F$ is an $\ell$-adic Lie group,
\item $F_{\infty}/F$ is unramified outside $\Sigma$,
\item $F_\infty$ contains $\FF_{q^{\ell^{\infty}}}F$.
\end{enumerate}

Assume that $\ell\neq p$ and let $\Z_{\ell}(1)$ be the $\Z_{\ell}$-sheaf on $X$ corresponding to the cyclotomic character $\Gal(\bar{\FF}_q/\FF_q)\mto \Z_{\ell}^{\times}$. We set
$$
C(F_{\infty}/F)=C(Y/X,\Z_{\ell}(1))[-3].
$$
Poitou-Tate duality then implies
$$
C(F_{\infty}/F)=\RDer\Hom(\RDer\Sect_{\et}(Y,\Q_{\ell}/\Z_{\ell}),\Q_{\ell}/\Z_{\ell}),
$$
in perfect accordance with the number field case \cite[\S 2.3]{Kakde2}. Since $Y$ is $K(\pi,1)$, one may replace the appearance of \'etale cohomology of $Y$ by the Galois cohomology of $\fundG(Y,\xi)$ if one desires.

One of the major differences between number fields and function fieds is that $\Spec \Z$ has no sensible compactification in the category of schemes. In particular, the standard construction of the total derived section functor with proper support does not work. However, the above duality statement explains why the complex $C(F_{\infty}/F)$ that appears in Kakde's work is a sensible replacement.

As in the number field case, we let $\X(F_\infty)$ denote the Galois group of the maximal abelian $\ell$-extension of $F_{\infty}$ unramified outside $\Sigma$. A quick calculation then shows
\begin{align*}
\HF^{-1}(C(F_{\infty}/F))&=\X(F_\infty),\\
\HF^{0}(C(F_{\infty}/F))&=\Z_{\ell},\\
\HF^{i}(C(F_{\infty}/F))&=0\quad\text{otherwise.}
\end{align*}

Beware that the module $\X(F_{\infty})$ in itself may have no finite projective resolution and therefore, no well-defined class in $K_0(\Lambda(G),\Lambda(G)_S)$ if $G$ has elements of order $\ell$. If we additionally assume that such elements do not exist, then
$$
[C(F_{\infty}/F)]=[\Z_{\ell}]-[\X(F_\infty)]
$$
in $K_0(\Lambda(G),\Lambda(G)_S)$ and Theorem~\ref{thm:mythm} gives the precise equivalent of Theorem~\ref{thm:mc number field} plus the vanishing of the $\mu$-invariant.

A precise equivalent of the non-commutative Iwasawa main conjecture for elliptic curves over number fields \cite{CFKSV} can be deduced from Theorem~\ref{thm:mythm} by considering the $\Z_\ell$-sheaf on $X$ given by the $\ell$-adic Tate module ($\ell\neq p$) of an elliptic curve or, more generally, any abelian variety over $F$. More details are given in \cite{Witte:FunctionFields}.

The case $\ell=p$ looks more difficult and seems very different in nature. Theorem~\ref{thm:Burnsthm} may be applied to the constant sheaf $\Z_p$ on $X$ to deduce a main conjecture for the leading terms of the Artin $L$-functions $L^A(\rho,s)$ in $s=0$. The cohomology of the complex $C(Y/X,\Z_p)$ is concentrated in degree $2$ and its $\Lambda(G)$-dual is related to the inverse limit of the $p$-parts of the class groups of the intermediate fields \cite[Prop.~4.1]{Burns:Congruences}. However, we cannot apply Burns' theorem to obtain interpolations of the leading terms in $s=n$ for arbitrary $n$ as the natural constructions of the Tate twists $\Z_p(n)$ via Bloch's cycle complexes do not give $\Z_p$-sheaves in our sense.  Moreover, the formulas for the leading terms of the zeta functions deduced by Milne \cite{Mil:VZFVFF} and others hint that there is a 'tangent space' contribution from the DeRham-complex which should appear in the boundary term of a non-commutative $p$-adic $L$-function at non-zero Tate twists. An equivariant generalisation of Milne's formulas is still missing.

Completing the work of Ochiai and Trihan \cite{OchiaiTrihan:OnTheSelmer}, a noncommutative main conjecture for abelian varieties over $F$ in the case $\ell=p$ is formulated in \cite{TrihanVauclair} and there is progress towards a proof, at least under certain hypotheses. The conjectured interpolation property is only for the leading terms of the $L$-functions of Artin twists of the abelian variety in $s=1$. Instead of \'etale cohomology the authors use flat cohomology as in \cite{KT:CBSDCP}. The boundary of the conjectured non-commutative $L$-function is given by the flat total derived section complex of the (flat) $p$-adic Tate module plus a tangent space contribution given by the total derived sections of the Lie algebra of the abelian variety. One might be tempted to apply Theorem~\ref{thm:Burnsthm} to the $\Z_p$-sheaf on $X$ given by the (\'etale) $p$-adic Tate module, but the $L$-function of this sheaf differs from the $L$-function of the abelian variety and it is not clear how to handle the difference.

Fixing a place $\place{p}$ of the function field $F$, one may speculate about yet another, completely different, approach to formulate an analogue of the Iwasawa main conjecture, with the valuation ring $A_{\place{p}}$ of $F_{\place{p}}$ taking over the role of $\Z_p$.  One can define a characteristic $p$ version of an $L$-function taking values in the completion of the algebraic closure of $F_{\place{p}}$ \cite{Goss} and one can prove formulas for special values of  characteristic $p$ $L$-functions \cite{Lafforgue}. According to the general philosophy, a possible analogue of the Iwasawa main conjecture should then give information about the limit of these values if $F$ varies in a suitable family of field extensions. Such a suitable family might be given by the extensions obtained from Drinfeld-Hayes modules. These are widely regarded as the right analogues of cyclotomic extensions in this setting.

\bibliographystyle{amsalpha}
\bibliography{Literature}
\end{document}